\documentclass[12pt]{amsart}
\usepackage{}
\usepackage{amsmath}
\usepackage{amsfonts}
\usepackage{amssymb}
\usepackage[all]{xy}           
\usepackage{extarrows}

\usepackage{bbding}
\usepackage{txfonts}
\usepackage[shortlabels]{enumitem}
\usepackage{ifpdf}
\ifpdf
  \usepackage[colorlinks,final,backref=page,hyperindex]{hyperref}
\else
  \usepackage[colorlinks,final,backref=page,hyperindex,hypertex]{hyperref}
\fi
\usepackage{tikz}
\usepackage[active]{srcltx}

\newcommand{\deleted}[1]{}
\newcommand{\delete}[1]{}
\newcommand{\mynote}[1]{}
\newcommand\notes[1]{}

\newcommand\changed[1]{#1}
\changed{}

\newtheorem{theorem}{Theorem}[section]

\newtheorem{lemma}[theorem]{Lemma}
\newtheorem{coro}[theorem]{Corollary}

\newtheorem{prop}[theorem]{Proposition}
\theoremstyle{definition}
\newtheorem{defn}[theorem]{Definition}

\newtheorem{exam}[theorem]{Example}
\deleted{\newtheorem{theorem}{Theorem}[section]
\newtheorem{prop-def}{Proposition-Definition}[section]
\newtheorem{coro-def}{Corollary-Definition}[section]

}

\newcommand{\nc}{\newcommand}
\nc{\tred}[1]{\textcolor{red}{#1}}
\nc{\tblue}[1]{\textcolor{blue}{#1}}
\nc{\tgreen}[1]{\textcolor{green}{#1}}
\nc{\tpurple}[1]{\textcolor{purple}{#1}}
\nc{\btred}[1]{\textcolor{red}{\bf #1}}
\nc{\btblue}[1]{\textcolor{blue}{\bf #1}}
\nc{\btgreen}[1]{\textcolor{green}{\bf #1}}
\nc{\btpurple}[1]{\textcolor{purple}{\bf #1}}

\renewcommand{\Bbb}{\mathbb}

\newcommand{\efootnote}[1]{}

\renewcommand{\textbf}[1]{}

\nc{\mlabel}[1]{\label{#1}}  
\nc{\mcite}[1]{\cite{#1}}  
\nc{\mref}[1]{\ref{#1}}  
\nc{\mbibitem}[1]{\bibitem{#1}} 

\delete{
\nc{\mlabel}[1]{\label{#1}  
{\hfill \hspace{1cm}{\bf{{\ }\hfill(#1)}}}}
\nc{\mcite}[1]{\cite{#1}{{\bf{{\ }(#1)}}}}  
\nc{\mref}[1]{\ref{#1}{{\bf{{\ }(#1)}}}}  
\nc{\mbibitem}[1]{\bibitem[\bf #1]{#1}} 
}

\renewcommand\geq{\geqslant}
\renewcommand\leq{\leqslant}

\renewcommand\bar[1]{\overline{#1}}
\renewcommand\tilde[1]{\widetilde{#1}}

\nc{\lead}{\mathrm{Lead}}
\nc{\Id}{\mathrm{Id}}
\nc{\Irr}{\mathrm{Irr}}
\nc{\End}{\mathrm{End}}
\nc{\vx}{\sigma}
\nc{\vy}{\tau}
\nc{\dvx}{\sigma^{(1)}}
\nc{\dvy}{\tau^{(1)}}
\nc{\done}{\vep}
\nc{\citep}[1]{\cite{#1}}
\nc{\wt}{\mathrm{wt}}
\nc{\bre}[1]{|#1|} \nc{\mapmonoid}{\frakM}
\nc{\disjoint}{\frakM'}
\nc{\ncpoly}[1]{\langle #1\rangle}  
\nc{\mapm}[1]{\lfloor\!|{#1}|\!\rfloor}             
\nc{\diff}[1]{{}^\NC\{ #1 \}}
\nc{\disj}[1]{\{{#1}\}'}
\nc{\mdisj}[1]{\frakM'(#1)}
\nc{\brho}{\bar{\rho}}
\nc{\om}{\bar{\frakm}}
\nc{\frakn}{\mathfrak n}
\nc{\ddeg}[1]{^{(#1)}}
\nc{\opset}{X} \nc{\genset}{{Z}}
\nc{\NC}{\mathrm{{NC}}}
\nc{\leaf}{\mathrm{leaf}}
\nc{\twig}{\mathrm{twig}}
\nc{\fe}{\mathrm{fl}}
\nc{\munderline}[1]{#1}
\nc{\bo}{o}
\nc{\ofe}{\mathrm{ofl}}
\nc{\dfe}{\mathrm{dfe}}
\nc{\fex}{\mathrm{fex}}
\nc{\dl}{\mathrm{dlex}}
\nc{\db}{\mathrm{db}}

\nc{\bin}[2]{ (_{\stackrel{\scs{#1}}{\scs{#2}}})}  
\nc{\binc}[2]{ \left (\!\! \begin{array}{c} \scs{#1}\\
    \scs{#2} \end{array}\!\! \right )}  
\nc{\bincc}[2]{  \left ( {\scs{#1} \atop
    \vspace{-1cm}\scs{#2}} \right )}  
\nc{\bs}{\bar{S}}
\nc{\cosum}{\sqsubset}
\nc{\la}{\longrightarrow}
\nc{\rar}{\rightarrow}
\nc{\dar}{\downarrow}
\nc{\dprod}{**}
\nc{\dap}[1]{\downarrow \rlap{$\scriptstyle{#1}$}}
\nc{\md}{\mathrm{dth}}
\nc{\uap}[1]{\uparrow \rlap{$\scriptstyle{#1}$}}
\nc{\defeq}{\stackrel{\rm def}{=}}
\nc{\disp}[1]{\displaystyle{#1}}
\nc{\dotcup}{\ \displaystyle{\bigcup^\bullet}\ }
\nc{\gzeta}{\bar{\zeta}}
\nc{\hcm}{\ \hat{,}\ }
\nc{\hts}{\hat{\otimes}}
\nc{\barot}{{\otimes}}
\nc{\free}[1]{\bar{#1}}
\nc{\uni}[1]{\tilde{#1}}
\nc{\hcirc}{\hat{\circ}}
\nc{\leng}{\ell}
\nc{\lleft}{[}
\nc{\lright}{]}
\nc{\lc}{\lfloor}
\nc{\rc}{\rfloor}
\nc{\lb}{[} 
\nc{\rb}{]} 
\nc{\curlyl}{\left \{ \begin{array}{c} {} \\ {} \end{array}
    \right .  \!\!\!\!\!\!\!}
\nc{\curlyr}{ \!\!\!\!\!\!\!
    \left . \begin{array}{c} {} \\ {} \end{array}
    \right \} }
\nc{\longmid}{\left | \begin{array}{c} {} \\ {} \end{array}
    \right . \!\!\!\!\!\!\!}
\nc{\onetree}{\bullet}
\nc{\ora}[1]{\stackrel{#1}{\rar}}
\nc{\ola}[1]{\stackrel{#1}{\la}}
\nc{\ot}{\otimes}
\nc{\mot}{{{\boxtimes\,}}}
\nc{\otm}{\overline{\boxtimes}}
\nc{\sprod}{\bullet}
\nc{\scs}[1]{\scriptstyle{#1}}
\nc{\mrm}[1]{{\rm #1}}
\nc{\msum}{\sum\limits}
\nc{\margin}[1]{\marginpar{\rm #1}}   
\nc{\dirlim}{\displaystyle{\lim_{\longrightarrow}}\,}
\nc{\invlim}{\displaystyle{\lim_{\longleftarrow}}\,}
\nc{\mvp}{\vspace{0.3cm}}
\nc{\tk}{^{(k)}}
\nc{\tp}{^\prime}
\nc{\ttp}{^{\prime\prime}}
\nc{\svp}{\vspace{2cm}}
\nc{\vp}{\vspace{8cm}}
\nc{\proofbegin}{\noindent{\bf Proof: }}
\nc{\proofend}{$\blacksquare$ \vspace{0.3cm}}
\nc{\modg}[1]{\!<\!\!{#1}\!\!>}
\nc{\intg}[1]{F_C(#1)}
\nc{\lmodg}{\!<\!\!}
\nc{\rmodg}{\!\!>\!}
\nc{\cpi}{\widehat{\Pi}}
\nc{\sha}{{\mbox{\cyr X}}}  
\nc{\shap}{{\mbox{\cyrs X}}} 
\nc{\shpr}{\diamond}    
\nc{\shp}{\ast}
\nc{\shplus}{\shpr^+}
\nc{\shprc}{\shpr_c}    
\nc{\msh}{\ast}
\nc{\zprod}{m_0}
\nc{\oprod}{m_1}
\nc{\vep}{\varepsilon}
\nc{\labs}{\mid\!}
\nc{\rabs}{\!\mid}

\nc{\dth}{d}
\nc{\mmbox}[1]{\mbox{\ #1\ }}
\nc{\fp}{\mrm{FP}} \nc{\rchar}{\mrm{char}} \nc{\Fil}{\mrm{Fil}}
\nc{\Mor}{Mor\xspace}
\nc{\gmzvs}{gMZV\xspace}
\nc{\gmzv}{gMZV\xspace}
\nc{\mzv}{MZV\xspace}
\nc{\mzvs}{MZVs\xspace}
\nc{\Hom}{\mrm{Hom}} \nc{\id}{\mrm{id}} \nc{\im}{\mrm{im}}
\nc{\incl}{\mrm{incl}} \nc{\map}{\mrm{Map}} \nc{\mchar}{\rm char}
\nc{\nz}{\rm NZ} \nc{\supp}{\mathrm Supp}

\nc{\Alg}{\mathbf{Alg}}
\nc{\Bax}{\mathbf{Bax}}
\nc{\bff}{\mathbf f}
\nc{\bfk}{{\bf k}}
\nc{\bfone}{{\bf 1}}
\nc{\bfx}{\mathbf x}
\nc{\bfy}{\mathbf y}
\nc{\base}[1]{\bfone^{\otimes ({#1}+1)}} 
\nc{\Cat}{\mathbf{Cat}}
\delete{}
\nc{\detail}{\marginpar{\bf More detail}
    \noindent{\bf Need more detail!}
    \svp}
\nc{\Int}{\mathbf{Int}}
\nc{\Mon}{\mathbf{Mon}}
\nc{\rbtm}{{shuffle }}
\nc{\rbto}{{Rota-Baxter }}
\nc{\remarks}{\noindent{\bf Remarks: }}
\nc{\Rings}{\mathbf{Rings}}
\nc{\Sets}{\mathbf{Sets}}

\nc{\BA}{{\Bbb A}} \nc{\CC}{{\Bbb C}} \nc{\DD}{{\Bbb D}}
\nc{\EE}{{\Bbb E}} \nc{\FF}{{\Bbb F}} \nc{\GG}{{\Bbb G}}
\nc{\HH}{{\Bbb H}} \nc{\LL}{{\Bbb L}} \nc{\NN}{{\Bbb N}}
\nc{\KK}{{\Bbb K}} \nc{\QQ}{{\Bbb Q}} \nc{\RR}{{\Bbb R}}
\nc{\TT}{{\Bbb T}} \nc{\VV}{{\Bbb V}} \nc{\ZZ}{{\Bbb Z}}

\nc{\cala}{{\mathcal A}} \nc{\calc}{{\mathcal C}}
\nc{\cald}{{\mathcal D}} \nc{\cale}{{\mathcal E}}
\nc{\calf}{{\mathcal F}} \nc{\calg}{{\mathcal G}}
\nc{\calh}{{\mathcal H}} \nc{\cali}{{\mathcal I}}
\nc{\call}{{\mathcal L}} \nc{\calm}{{\mathcal M}}
\nc{\caln}{{\mathcal N}} \nc{\calo}{{\mathcal O}}
\nc{\calp}{{\mathcal P}} \nc{\calr}{{\mathcal R}}
\nc{\cals}{{\mathcal S}}
\nc{\calt}{{\mathcal T}} \nc{\calw}{{\mathcal W}}
\nc{\calk}{{\mathcal K}} \nc{\calx}{{\mathcal X}}
\nc{\CA}{\mathcal{A}}

\nc{\fraka}{{\mathfrak a}} \nc{\frakA}{{\mathfrak A}}
\nc{\frakb}{{\mathfrak b}} \nc{\frakB}{{\mathfrak B}}
\nc{\frakD}{{\mathfrak D}} \nc{\frakH}{{\mathfrak H}}
\nc{\frakM}{{\mathfrak M}} \nc{\bfrakM}{\overline{\frakM}}
\nc{\frakm}{{\mathfrak m}} \nc{\frakP}{{\mathfrak P}}
\nc{\frakN}{{\mathfrak N}} \nc{\frakp}{{\mathfrak p}}
\nc{\frakS}{{\mathfrak S}} \nc{\frakx}{{\mathfrak x}}
\nc{\ox}{\bar{\frakx}} \nc{\frakX}{{\mathfrak X}}
\nc{\fraky}{{\mathfrak y}} \nc\dop{\delta}

\font\cyr=wncyr10
\font\cyrs=wncyr7
\nc{\redt}[1]{\textcolor{red}{#1}}

\nc{\gl}[1]{\textcolor{red}{\tt \underline{GL:}#1}}
\nc{\ql}[1]{\textcolor{blue}{\tt \underline{QL:}#1}}
\nc{\xing}[1]{\textcolor{purple}{\tt \underline{Xing:}#1}}

\nc{\rbo}{R_{RB}\langle Q \rangle}
\nc{\rbeta}{(R,P)} \nc{\salpha}{(S,\alpha)} \nc{\tgamma}{(T,\gamma)}
\nc{\otr}{\ot_{\rbeta}} \nc{\rhom}{{\rm Hom}}
\nc{\Free}{F_1} \nc{\rf}{\tilde{F}}
\nc{\revise}[1]{\textcolor{red}{#1}}

\begin{document}
\title{Rota-Baxter modules toward derived functors}

\author{Xing Gao}
\address{School of Mathematics and Statistics, Key Laboratory of Applied Mathematics and Complex Systems, Lanzhou
University, Lanzhou, Gansu 730000, P.R. China}
\email{gaoxing@lzu.edu.cn}

\author{Li Guo}
\address{
    Department of Mathematics and Computer Science,
         Rutgers University,
         Newark, NJ 07102, USA}
\email{liguo@rutgers.edu}

\author{Li Qiao}
\address{School of Mathematics and Statistics, Key Laboratory of Applied Mathematics and Complex Systems, Lanzhou
University, Lanzhou, Gansu 730000, P.R. China}
\email{liqiaomail@126.com}


\date{\today}

\begin{abstract}
In this paper we study Rota-Baxter modules with emphasis on the role played by the Rota-Baxter operators and resulting difference between Rota-Baxter modules and the usual modules over an algebra. We introduce the concepts of free, projective, injective and flat Rota-Baxter modules. We give the construction of free modules and show that there are enough projective, injective and flat Rota-Baxter modules to provide the corresponding resolutions for derived functor.
\end{abstract}

\subjclass[2010]{16W99,16S10,16T10,16T30}

\keywords{
Rota-Baxter algebra; Rota-Baxter module; free module; projective module;  injective module; flat module; ring of Rota-Baxter operators}

\date{\today}

\maketitle

\tableofcontents

\setcounter{section}{0}


\section{Introduction}

Motivated by his probability study~\mcite{Ba}, G. Baxter introduced the concept of a (Rota-)Baxter algebra in 1960. To recall its definition, let $\bfk$ be a commutative ring with identity $\bfone_{\bfk}$ and fix a $\lambda\in \bfk$.
A Rota-Baxter algebra of weight $\lambda$ is a pair $(R,P)$ where $R$ is an algebra and $P$ is a linear operator on $R$ satisfying the {\bf Rota-Baxter axiom}
\begin{equation}
P(r)P(s)=P(rP(s))+P(P(r)s)+\lambda P(rs) \quad \text{for all } r, s\in R.
\mlabel{eq:rba}
\end{equation}
In the 1960s through 1990s, this algebraic structure was studied from analytic and combinatorial viewpoints with contributions from well-known mathematicians such as Atkinson, Cartier and Rota~\mcite{At,Ca,Ro1,Ro2}. In the Lie algebra context, it was related to the operator form of the classical Yang-Baxter equation by the Russian physicists~\mcite{STS}. Since the beginning of this century, this area has experienced a burst of development with broad applications ranging from number theory to quantum field theory~\mcite{Ag,BCQ,Bai,CK,EGM,GK1,GZ}. See~\mcite{Gus} for a survey and ~\mcite{Gub} for a more detailed treatment.

Representation theory is an important aspect in the study of any algebraic structure. A representation of a Rota-Baxter algebra is made more involved because of the Rota-Baxter operator $P$ on top of the algebra $R$. As introduced in~\mcite{GL} (see also~\mcite{LQ}), a (left) Rota-Baxter module is defined to be a (left) $R$-module $M$ together with a linear operator $p$ on $M$ which satisfies the module form of Eq.~(\mref{eq:rba}):
\begin{equation}
P(r)p(m) = p(rp(m))+p(P(r)m)+\lambda p(rm) \quad \text{for all } r\in R, m\in M.
\mlabel{eq:rbm}
\end{equation}
Note that for any $\bfk$-algebra $R$ and $\lambda\in \bfk$, the scalar product operator
$$ R\longrightarrow R, \quad r\mapsto -\lambda r \quad \text{for all } r\in R,$$
is a Rota-Baxter operator of weight $\lambda$. Thus any $\bfk$-algebra can be naturally regarded as a Rota-Baxter algebra of weight $\lambda$. Likewise, any $R$-module with the same scalar product operator is a Rota-Baxter module. Thus the study of Rota-Baxter modules generalizes the study of the usual modules.

In this paper, we study Rota-Baxter modules as a first step to study their homological algebra. Thus we study the free, projective, injective and flat objects in the category of Rota-Baxter modules. We show that there are enough of these objects in this category, enabling us to define the derived functors in the category of Rota-Baxter modules.

As observed in~\mcite{GL}, a Rota-Baxter module can be regarded as a module on the ring of Rota-Baxter operators on the Rota-Baxter algebra. Since the ring of Rota-Baxter operators in general is not yet well-understood, it is useful to study Rota-Baxter modules via a direct approach as we are taking in this paper. Further, this approach makes it easier to see the difference between Rota-Baxter modules and the usual modules.
For example, a right Rota-Baxter module needs to be defined by an identity different from Eq.~(\mref{eq:rbm}), imposing a particularly strong condition for a Rota-Baxter algebra to be a right module or a bimodule over itself. See Propositions~\mref{pp:rrbm} and~\mref{pp:rbbm}.
Further, contrary to the fact that an algebra is a free module over itself, a Rota-Baxter algebra is not a free Rota-Baxter module over itself, but satisfies a universal property in a restricted sense. See Theorem~\mref{thm:fmrbm}.
Overall, even though the concepts for Rota-Baxter modules can be defined in analogue to those for modules, their constructions needs new ingredients.

Here is an outline of the paper. In Section~\mref{sec:free}, we first introduce basic notations on Rota-Baxter modules, emphasizing the difference between a right Rota-Baxter module and a left one. We then construct free operated modules and then utilize them to obtain free Rota-Baxter modules by taking quotients. Further a usual free module is characterized as a free Rota-Baxter module with an additional restriction. In Section~\mref{sec:pi} the concepts of a projective Rota-Baxter module and an injective Rota-Baxter module are defined. It is shown that there are enough projective and injective Rota-Baxter modules to obtain projective and injective resolutions of a Rota-Baxter module, allowing the definition of Rota-Baxter homology and cohomology groups. In Section~\mref{sec:flat}, the concept of a tensor product over a Rota-Baxter algebra is introduced from which a flat Rota-Baxter module is defined. It is shown that free and more generally projective Rota-Baxter modules are flat Rota-Baxter modules.

Throughout the paper, all algebras, linear maps and tensor products are taken over the base ring $\bfk$ unless otherwise stated.

\section{Free Rota-Baxter modules}
\mlabel{sec:free}
After introducing basic notions on Rota-Baxter modules, emphasizing their difference from modules over an algebra, we give a construction of free Rota-Baxter modules through operated modules.

\subsection{Rota-Baxter modules}
We first recall the notion of left Rota-Baxter modules from~\cite{GL,LQ} before introducing the different notion of right Rota-Baxter modules.
\begin{defn}
Let $(R,P)$ be a Rota-Baxter $\bf k$-algebras of weight $\lambda$.
\begin{enumerate}
\item A {\bf left (Rota-Baxter) $(R,P)$-module} $(M,p)$ is a
left $R$-module $M$ together with a \bfk-linear operator $p: M \longrightarrow M$ such that
\begin{equation}
P(r)p(m)=p(P(r)m) + p(rp(m))+\lambda p(rm) \quad \text{for all } r\in R, m\in M. \mlabel{eq:lrbm}
\end{equation}
\item For left $(R,P)$-modules $(M,p)$ and $(M',p')$, a {\bf left $(R,P)$-module homomorphism} is a left $R$-module homomorphism $\phi: M \longrightarrow M'$ such that $\phi \circ p=p' \circ \phi$.
\item A {\bf left $(R,P)$-module monomorphism (resp. epimorphism, isomorphism)} is defined to be an injective (resp. surjective, bijective) left $(R,P)$-module homomorphism.
\item A {\bf left $\rbeta$-submodule} of $(M,p)$ is a submodule $N$ of $R$-module $M$ such that $p(N)\subseteq N$, giving the pair $(N,p|_N)$.
\end{enumerate}
\end{defn}

To define a quotient module of a left Rota-Baxter module, we have
\begin{lemma}
Let $(R,P)$ be a Rota-Baxter \bfk-algebra of weight $\lambda$, $(M,p)$ a left $(R,P)$-module and $(N,p|_N)$ a submodule of $(M,p)$. The pair
$(M/N,\bar{p})$ is a left $\rbeta$-module, where
$$\bar{p}: M/N \longrightarrow M/N, \ m+N \longmapsto p(m)+N.$$
We call $(M/N, \bar{p})$ the {\bf quotient $\rbeta$-module} of $(M,p)$ by $(N, p|_N)$.
\mlabel{lem:quotmod}
\end{lemma}

\begin{proof}
Since $p(N)\subseteq N$, the prescription of $\overline{p}$ is well-defined. Next, we verify that $\bar{p}$ satisfies Eq.~(\mref{eq:lrbm}).
For any $r\in R$, $m\in M$, we have
\begin{align*}
 P(r)\overline{p}(m+N) &=P(r)(p(m)+N) \\
 &= P(r)p(m) + N \\
 &= p(P(r)m)+p(rp(m))+\lambda p(rm) + N \\
 &= (p(P(r)m)+N) + (p(rp(m))+N) + (\lambda p(rm) + N) \\
 &= \overline{p}(P(r)(m+N))+\overline{p}(r\overline{p}(m+N))+\lambda \overline{p}(r(m+N)),
\end{align*}
as required.
\end{proof}

Denote by ${\bf _{(R,P)}Mod}$ the category of left $(R,P)$-modules,
with its objects the left $(R,P)$-modules and its morphisms the $(R,P)$-module homomorphisms.

The following are some examples of left Rota-Baxter modules.

\begin{exam}
Let $(R,P)$ be a Rota-Baxter \bfk-algebra of weight $\lambda$. Then
\begin{enumerate}
\item With $(R,P)$ acting on itself on the left, $(R,P)$ is a left $\rbeta$-module.

\item As in the case of the usual module theory over an algebra, any left Rota-Baxter ideal $I$ of $(R,P)$
(meaning a left ideal $I$ of $R$ such that $P(I)\subseteq I$) together with the restriction $P: I \longrightarrow I$ is a left
 $(R,P)$-module. Then $(R/I, \bar{P})$ is also a left
 $(R,P)$-module by Lemma~\mref{lem:quotmod}.

 \item Let $R[x]$ be the ring of polynomials with coefficients in $R$. Define
\begin{equation*}
\calp: R[x] \longrightarrow R[x], \ \sum_{i=0}^{n} c_{i} x^{i} \longmapsto \sum_{i=0}^{n} P(c_{i}) x^{i}.
\mlabel{eq:rbpo}
\end{equation*}
Then $(R[x],\calp)$ is a Rota-Baxter \bfk-algebra of weight $\lambda$ \mcite{Ro1,Gub}, and $(R[x],\calp)$ is a Rota-Baxter left $(R,P)$-module.
\end{enumerate}
\end{exam}

The difference between a module over an algebra and a Rota-Baxter module can already be observed by the concept of a Rota-Baxter right module.

\begin{defn}
Let $(R,P)$ be a Rota-Baxter algebra of weight $\lambda$.
A {\bf (Rota-Baxter) right $(R,P)$-module} $(M,p)$ is a
right $R$-module $M$ together with a \bfk-linear operator $p: M \longrightarrow M$ such that
\begin{equation}
p(mP(r))=p(m)P(r) + p(p(m)r)+\lambda p(m)r \quad \text{for all } r\in R, m\in M. \mlabel{eq:rrbm}
\end{equation}
\mlabel{defn:rbbm}
\end{defn}
A right $(R,P)$-module homomorphism is defined similarly to that for left $(R,P)$-modules.

The quite unorthodox definition of a Rota-Baxter right module originates from the Rota-Baxter operator and can be justified as follows. See Proposition~\ref{prop:bihom} for an application of Rota-Baxter right modules.

Taking left multiplications by elements of $R$, as well as the action of $p$, as linear operators in $\End(M)$, then Eq.~(\mref{eq:lrbm}):
$$P(r)p(m)=p(P(r)m)+p(r p(m))+\lambda p(r m) \quad \text{for all } r\in R, m\in M,$$
can be rewritten as
$$(P(r)\circ p)(m)=(p\circ P(r))(m)+(p\circ r \circ p)(m)+\lambda (p\circ r) (m),$$
regarding $M$ as a left $\End(M)$-module.
Then the corresponding right $\End(M)$-action on $M$ is
$$
(m)(P(r)\circ p)=(m)(p\circ P(r))+(m)(p\circ r \circ p)+\lambda (m)(p\circ r),$$
acting from the left to the right,
which gives
$$
(m P(r))p=((m)p)P(r)+(((m)p)r)p+\lambda ((m)p)r).
$$
This is Eq.~(\mref{eq:rrbm}).

In particular, if Rota-Baxter algebra $(R,P)$ is viewed as a right Rota-Baxter module over itself, then it needs to satisfy
\begin{equation}
P(rP(s))=P(r)P(s) + P(P(r)s)+\lambda P(r)s \quad \text{for all } r, s\in R.
\mlabel{eq:rrbo}
\end{equation}

\begin{prop}
Let $(R,P)$ be a Rota-Baxter \bfk-algebra of weight $\lambda$. Then $(R,P)$ is a right $(R,P)$-module if and only if
the Rota-Baxter operator $P$ satisfies the relation
\begin{equation}
2P(P(r)s)+\lambda P(rs)+\lambda P(r)s=0 \quad \text{for all } r, s\in R.
\label{RB-right}
\end{equation}
\mlabel{pp:rrbm}
\end{prop}
\begin{proof}
This is because, under the assumption of Eq.~(\mref{eq:rba}), $P$ satisfies Eq.~(\mref{eq:rrbo}) if and only if it satisfies Eq.~(\mref{RB-right}).
\end{proof}

Applying this result, we next give an example of a Rota-Baxter algebra which is a right $(R,P)$-module, as well as a left $(R,P)$-module.

\begin{prop}
Let $R:=\bfk u_{0}\oplus\bfk u_{1}$. Equip it with the multiplication
where $u_{0}$ is the identity and $u_{1}^{2}=-\lambda u_{1}.$
Define a \bfk-linear operator
$$P: R\longrightarrow R,
\ u_{0}\longmapsto u_{1}, \ u_{1}\longmapsto -\lambda u_{1}.$$
Then $(R,P)$ is a Rota-Baxter \bfk-algebra satisfying Eq.~(\mref{RB-right}) and hence is a right $(R,P)$-module.
\mlabel{pp:rrbm}
\end{prop}
\begin{proof}
The cyclic $\bfk$-module $\bfk u_1$ with $u_1^2=-\lambda u_1$ is a nonunitary $\bfk$-algebra. Then $R$ is simply the unitarization of $\bfk u_1$.

We next verify that $P$ satisfies the Rota-Baxter axiom in Eq.~(\ref{eq:rba}). Since $P$ is \bfk-linear, we only need to check it for the basis elements.

For $r=s=u_{0}$, we have
\begin{eqnarray*}
P(u_{0})P(u_{0})=u_{1}u_{1}=-\lambda u_{1},
\end{eqnarray*}
agreeing with
$$P(u_{0}P(u_{0}))+P(P(u_{0})u_{0})+\lambda P(u_{0}u_{0})
=2P(u_{1})+\lambda P(u_{0})=-2\lambda u_{1}+\lambda u_{1}=-\lambda u_{1}.
$$

For $r=u_{0}$, $s=u_{1}$, we have
$$
P(u_{0})P(u_{1})=u_{1}(-\lambda u_{1})=-\lambda u_{1}^{2}=\lambda^{2} u_{1}
$$
which agrees with
$$
P(u_{0}P(u_{1}))+P(P(u_{0})u_{1})+\lambda P(u_{0}u_{1})
=P(-\lambda u_{1})+P(u_{1}^2)+\lambda P(u_{1})
=P(u_1^2)=-\lambda P(u_{1})=\lambda^{2}u_{1}.
$$

For $r=s=u_{1}$, we have
\begin{eqnarray*}
P(u_{1})P(u_{1})=(-\lambda u_{1})(-\lambda u_{1})=\lambda^{2} u_{1}^{2}=-\lambda^{3} u_{1},
\end{eqnarray*}
agreeing with
$$
P(u_{1}P(u_{1}))+P(P(u_{1})u_{1})+\lambda P(u_{1}u_{1})
=2P(u_{1}(-\lambda u_{1}))+\lambda P(u_1^2)
=-\lambda P(u_1^2)=\lambda^2P(u_{1})=-\lambda^{3}u_{1}.
$$
Thus $(R,P)$ is a Rota-Baxter algebra of weight $\lambda$.

We finally verify that $P$ satisfies Eq.~(\ref{RB-right}).
Taking $r$ and $s$ to be the basis elements $u_0$ or $u_1$, we obtain
$$
2P(P(u_{0})u_{0})+\lambda P(u_{0}u_{0})+\lambda P(u_{0})u_{0}
=2P(u_{1})+\lambda u_{1}+\lambda u_{1}
=-2\lambda u_{1}+\lambda u_{1}+\lambda u_{1}=0.
$$
$$
2P(P(u_{0})u_{1})+\lambda P(u_{0}u_{1})+\lambda P(u_{0})u_{1}
=2P(u_{1}^{2})+\lambda P(u_{1})+\lambda u_{1}^2
=2P(-\lambda u_{1})-2\lambda^{2} u_{1}
=2\lambda^{2}u_{1}-2\lambda^{2} u_{1}=0.
$$
$$
2P(P(u_{1})u_{1})+\lambda P(u_{1}u_{1})+\lambda P(u_{1})u_{1}
=2P(-\lambda u_{1}^2)+\lambda P(-\lambda u_{1})-\lambda^2 u_{1}^2
=2\lambda^{2} P(u_{1})+2\lambda^{3} u_{1}
=0.
$$
Thus $P$ satisfies Eq.~(\ref{RB-right}).
\end{proof}

We next define Rota-Baxter bimodules.
\begin{defn}
Let $(R,P)$ and $(S,\alpha)$ be Rota-Baxter algebras. An {\bf $(R,P)$-$(S,\alpha)$-bimodule} is a triple~$(M, p_M^R, p_M^S)$ where $(M,p_M^R)$ is a left $\rbeta$-module, $(M,p_M^S)$ is a right $(S,\alpha)$-module and $M$ is an $R$-$S$-bimodule over algebras, such that
\begin{equation*}
p_M^S(rm) = r p_M^S(m),  \quad p_M^R(ms) = p_M^R(m)s, \quad
 p_M^S(p_M^R(m)) = p_M^R(p_M^S(m))
\mlabel{eq:rbbim}
\end{equation*}
for all $ m\in M, r\in R, s\in S.$
\end{defn}

In general, $(R,P)$ is not its own $(R,P)$-$(R,P)$-bimodule because being a Rota-Baxter bimodule implies that the operator $P$ is $R$-linear, but a Rota-Baxter operator is only \bfk-linear. Denote by $\bfone_{R}$ the identity of $R$. To be precise, we have

\begin{prop}
Let $(R,P)$ be a Rota-Baxter \bfk-algebra of weight $\lambda$. Then $(R,P)$ is an $(R,P)$-$(R,P)$-bimodule if $P$ is $R$-linear on both sides, and either $P(\bfone_{R})=0$ or $P(\bfone_{R})=-\lambda$. If $R$ has no zero divisors, then the converse is also true.
\mlabel{pp:rbbm}
\end{prop}
\begin{proof}
Suppose that the Rota-Baxter operator $P$ is $R$-linear and satisfies $P(\bfone_{R})=0$ or $P(\bfone_{R})=-\lambda$. Then $P$ is either the zero operator or the scalar operator $P(r)=-\lambda r$ and so $P^2(r)=\lambda^2 r$ for $r\in R$. If $P$ is the zero
operator, then everything vanishes in the conditions of a Rota-Baxter bimodule. So we are done. In the latter case, the check is also simple. For example, to check Eq.~(\ref{RB-right}), for every $r, s\in R$, we have
$$
2P(P(r)s)+\lambda P(rs)+\lambda (P(r)s)
=2rsP^2(\bfone_{R})+\lambda rsP(\bfone_{R})+\lambda rsP(\bfone_{R})
=2\lambda^2 rs-\lambda^2 rs-\lambda^2 rs=0.
$$

Conversely if $(R,P)$ is a $(R,P)$-$(R,P)$-bimodule. Then $P$ is $R$-linear by definition. Then by the Rota-Baxter axiom in Eq.~(\ref{eq:rba}) or Eq.~(\ref{RB-right}), we have $P(\bfone_{R})(P(\bfone_{R})+\lambda)=0$. Then the assumption that $R$ has no zero divisors implies $P(\bfone_{R})=0$ or $P(\bfone_{R})=-\lambda$.
\end{proof}

\subsection{Free operated modules}

We recall from~\mcite{Guop} that an {\bf operated \bfk-algebra} is a \bfk-algebra $R$ equipped with a \bfk-linear operator $\alpha:R\to R$.

\begin{defn}
Let $(R, \alpha)$ be an operated \bfk-algebra.
\begin{enumerate}
\item A {\bf left operated $R$-module} is a pair $(M,p)$ consisting of a left $R$-module $M$ and a $\bfk$-linear operator $p:M\longrightarrow M$.

\item Let $(M,p_M)$ and $(N,p_N)$ be left operated $R$-modules. A {\bf left operated $R$-module homomorphism} $f: (M,p_M) \longrightarrow (N,p_N)$
is a left $R$-module homomorphism $f:M \longrightarrow N$ such that $f\circ p_M = p_N \circ f$.
\end{enumerate}
\end{defn}

For example, left Rota-Baxter modules are left operated modules.

Following the convention made in the introduction, the tensor products are all taken over $\bfk$, unless otherwise stated. Let $(R,\alpha)$ be an operated \bfk-algebra and $X$ a set. Denote
\begin{align*}
\calm_R(X):= RX \oplus (R\ot R)X \oplus \cdots = \oplus_{n\geq 1}(R^{\otimes n}X)= (\oplus_{n\geq 1}R^{\otimes n})X,
\end{align*}
where $R^{\ot n}X=R^{\ot n}\ot \bfk X$ with $\bfk X$ being the free $\bfk$-module on $X$. The action of $R$ on the left most tensor factor of $R^{\ot n}X$ defines a left action of $R$ on $\calm_R(X)$, giving rise to a left $R$-module structure on $\calm_R(X)$. Define a \bfk-linear operator $p_X: \calm_R(X)\to \calm_R(X)$ by assigning
$$(r_1\ot \cdots \ot r_n)x\mapsto (\bfone_{R}\ot r_1\ot \cdots \ot r_n) x \quad \text{for all }  r_1, \cdots, r_n\in R, x\in X$$
for pure tensors $r_1\ot\cdots\ot r_n$ and extending by additivity.

\begin{prop} Let $(R,\alpha)$ be an operated\, \bfk-algebra and $X$ a set. Then, with the above notations,
\begin{enumerate}
\item the pair $(\calm_R(X), p_X)$ is a left operated $R$-module; \mlabel{it:om}
\item the pair $(\calm_R(X), p_X)$, together with the natural embedding $j_X: X\longrightarrow \calm_R(X)$, is the free left operated $R$-module generated by $X$.
More precisely, for any operated left $(R,\alpha)$-module $(M, q)$ and any set map $f: X\longrightarrow M$, there exists a unique
operated $R$-module homomorphism $\tilde{f}: \calm_R(X) \longrightarrow M$ such that $\tilde{f} \circ j_X = f$, that is the following diagram commutes.
$$
\xymatrix{
X\ar@{^{(}->}[rr]^{j_X} \ar^{f}[drr]&&(\calm_R(X),p_X)\ar@{-->}^{\tilde{f}}[d]&\\
&&(M,q) }
$$
\mlabel{it:fom}

\end{enumerate}
\mlabel{prop:freeo}
\end{prop}

\begin{proof}
(\mref{it:om})~~ We only need to verify that $p_{X}$ is $\bfk$-linear which follows from the \bfk-linearity of the tensor product:
$$\bfone_{R} \ot (kr_1)\ot \cdots \ot r_n = k \bfone_{R} \ot r_1\ot \cdots \ot r_n \quad \text{for all } k\in \bfk, r_1, \cdots, r_n\in R, n\geq 1.$$

\noindent
(\mref{it:fom})~~
We define $\tilde{f}: \calm_R(X) \to M$ by defining $\tilde{f}((r_1\ot \cdots \ot r_n)x)$ for $(r_1\ot \cdots \ot r_n)x\in R^{\ot n}X$ recursively on $n$. For the initial step of
$n=1$, we define
\begin{equation}
\tilde{f}(r_1 x) = r_1 \bar{f} (x) = r_1 (\tilde{f}\circ j_X) (x)= r_1 f(x). \mlabel{eq:fomho}
\end{equation}
For the induction step, we define
$$\bar{f}((r_1\ot \cdots \ot r_n)x) = r_1 q (\tilde{f}((r_2\ot \cdots \ot r_n)x)).$$
By construction, $\tilde{f}$ is a left $R$-module homomorphism. Note that this is also the only way to define $\tilde{f}$ under the conditions $\tilde{f} \circ j_X = f$ and $q\circ \tilde{f}=\tilde{f}\circ p_X$, proving the desired uniqueness of $\tilde{f}$ for the universal property.
\end{proof}

\subsection{Free Rota-Baxter module}

We now apply free operated modules to construct free Rota-Baxter modules.
\begin{defn}
Let $(R,P)$ be a Rota-Baxter \bfk-algebra of weight $\lambda$ and $X$ a set.
A {\bf free left $\rbeta$-module} on $X$ is a left $\rbeta$-module $(F(X),p)$ together with a map $j_X: X\longrightarrow F(X)$
satisfying the following universal property: for any left $\rbeta$--module $(M,q)$ and any set map $f: X\longrightarrow M$, there exists a unique
left $\rbeta$-module homomorphism $\widetilde{f}: F(X) \longrightarrow M$ such that $\widetilde{f} \circ j_X= f$.
\end{defn}

Now we construct free left $\rbeta$-modules.
Let $I_X$ denote the left operated submodule of $\calm_R(X)$ generated by the subset
$$
\{ P(r)p_X(y) - p_X(P(r)y) - p_X(rp_X(y)) - \lambda p_X(ry)~|~r\in R, y\in \calm_R(X)\}.
$$
Define $\calm_R(X)/I_X$ to be the quotient operated module of $\calm_R(X)$ by $I_X$ and define
$$p: \calm_R(X)/I_X \longrightarrow \calm_R(X)/I_X,\ y+ I_X \longmapsto p_X(y)+I_X$$
to be the operator on the quotient $\calm_R(X)/I_X$ induced by $p_X$. Then $(\calm_R(X)/I_X, p)$ is a left Rota-Baxter $\rbeta$-module.

\begin{theorem}
Let $(R,P)$ be a Rota-Baxter \bfk-algebra of weight $\lambda$ and $X$ a set.
Then $(\calm_R(X)/I_X, p)$ with the natural map $j:= \pi\circ j_X: X\longrightarrow \calm_R(X) \longrightarrow \calm_R(X)/I_X $ is the free left $\rbeta$-module.
\mlabel{thm:freem}
\end{theorem}

\begin{proof}
Let $(M,q)$ be a left $\rbeta$-module and $f: X\longrightarrow M$ a set map.
From Proposition~\mref{prop:freeo}, there is a unique left operated $R$-module homomorphism
$\tilde{f}: (\calm_R(X),p_X) \longrightarrow (M, q)$ such that $\tilde{f} \circ j_X = f,$
as showing in the following diagram:
$$
\xymatrix{
X\ar^{j_X}[rr] \ar_{f}[drr] &&(\calm_R(X), p_X)\ar^{\pi}[rr]  \ar^{\tilde{f}}[d]
&& (\calm_R(X)/I_X, p)\ar@{-->}^{\bar{f}}[dll] \\
&&(M,q)&&
}
$$
Now we show that $\tilde{f}$ vanishes on the generators of $I_X$. Indeed, let $r\in R$ and $y\in \calm_R(X)$. Then
\begin{align*}
& \tilde{f}(P(r)p_X(y) - p_X(P(r)y) - p_X(rp_X(y)) - \lambda p_X(ry))\\
=& P(r)q(\tilde{f}(y)) - q(P(r)\tilde{f}(y))  - q(rq(\tilde{f}(y))) - \lambda q(r\tilde{f}(y))\\
=&0.
\end{align*}
So $\tilde{f}$ induces a unique left $\rbeta$-module homomorphism
$$\bar{f}:(\calm_R(X)/I_X,p) \longrightarrow (M,q)$$
such that $\bar{f}\circ \pi = \tilde{f}$. Thus
$$\bar{f}\circ j = \bar{f}\circ \pi\circ j_X = \tilde{f}\circ j_X = f,$$
as required. The uniqueness of $\bar{f}$ follows from the uniqueness of $\tilde{f}$ and the uniqueness of its induced map on the quotient $\calm_R(X)/I_X$.
\end{proof}

As in the case of modules, we obtain
\begin{coro}
\begin{enumerate}
\item Every left Rota-Baxter module is the quotient of a free left Rota-Baxter module.  \mlabel{it:quof}

\item Every finitely generated left Rota-Baxter module is the quotient of a finitely generated free left Rota-Baxter module. \mlabel{it:quff}
\end{enumerate}
\mlabel{coro:quof}
\end{coro}

\subsection{Free modules as free Rota-Baxter modules}
As noted in the introduction, a Rota-Baxter algebra $(R,P)$ in general is not free as a Rota-Baxter module over itself. We now make this precise. We show that a Rota-Baxter algebra is a free Rota-Baxter module in a more restricted sense. More generally, we investigate how a free $R$-module behaves like a free Rota-Baxter module.

Let $\rbeta$ be a Rota-Baxter algebra of weight $\lambda$ and $X$ a set. For a left $(R,P)$-module $(M,q)$, set
$$
MC(M):=\{m\in M \mid q(rm)=P(r)m\, \text{ for any } r\in R\},
$$
called the set of {\bf module constants} of $M$ since $m\in MC(M)$ behaves like a constant which can be taken out of the operator. Let $\rf(X)$ be the free left $R$-module
generated by $X$:
$$\rf(X):=\left\{\sum_{x\in X}r_{x}x~\big|~r_{x}\in R\right\} .$$
Define
$$\tilde{p}: \rf(X) \longrightarrow \rf(X),\ \sum_{x\in X}r_{x}x \longmapsto \sum_{x\in X}P(r_{x})x.$$

\begin{theorem} Let $\rbeta$ be a Rota-Baxter algebra of weight $\lambda$ and $X$ a set. Then
\begin{enumerate}
\item the pair $(\rf(X),\tilde{p})$ is a left $(R, P)$-module. \mlabel{it:rrbm}

\item  the pair $(\rf(X),\tilde{p})$, together with the natural embedding map $\iota: X\longrightarrow (\rf(X), \tilde{p})$,
is the {\bf restricted free left $(R,P)$-module generated by $X$} in the sense that,
for any left $\rbeta$-module $(M,q)$ and any
set map $f: X\longrightarrow (M,q)$ with $\im f \subseteq MC(M)$, there exists a unique left Rota-Baxter homomorphism
$\bar{f}: (\rf(X),\tilde{p})\longrightarrow (M,q)$ such that $f=\bar{f}\circ \iota$.
\mlabel{it:frrbm}
\end{enumerate}
\mlabel{thm:fmrbm}
\end{theorem}

\begin{proof}
(\mref{it:rrbm}) It is sufficient to show that $\tilde{p}$ satisfies Eq.~(\mref{eq:lrbm}).
For any $r\in R$ and $r_x x\in F(X)$, we have
\begin{align*}
 P(r)\tilde{p}( r_{x}x) &= P(r)( P(r_{x})x)\\
 &= P(r)P(r_{x})x \\
 &=  P(rP(r_{x}))x + P(P(r)r_{x})x + \lambda P(rr_{x})x \\
 &=  \tilde{p}(rP(r_{x})x) + \tilde{p}(P(r)r_{x}x) + \lambda \tilde{p}(rr_{x}x) \\
 &= \tilde{p}(r \tilde{p}( r_{x}x)) + \tilde{p}(P(r)( r_{x}x)) + \lambda \tilde{p}(r( r_{x}x)),
\end{align*}
as required.
\smallskip

\noindent
(\mref{it:frrbm}) By the universal property of $\rf(X)$ as the free left $R$-module over $X$, there is a left $R$-module homomorphism
\begin{equation}
\bar{f}:(\rf(X), \tilde{p})\longrightarrow (M,q), \ \sum_{x\in X}r_{x}x \longmapsto \sum_{x\in X}r_{x}f(x).
\mlabel{eq:alghom1}
\end{equation}
Furthermore,
\begin{align*}
(\bar{f}\circ \tilde{p}) (r_{x}x)  = & \bar{f} (\tilde{p}( r_{x}x)) = \bar{f} ( P(r_{x})x)= P(r_{x})f(x)=q(r_{x}f(x)) = q(\bar{f}( r_{x}x)) = (q\circ \bar{f})( r_{x}x),
\end{align*}
where the third step follows from Eq.~(\mref{eq:alghom1}) and the fourth step from $\im f \subseteq MC(M)$. Thus $\bar{f} \circ \tilde{p} = q\circ \bar{f}$ and so $\bar{f}$
is the desired left $\rbeta$-module homomorphism.

By the definition of $\bar{f}$, we have
$$\bar{f}(r_{x}x) =   r_x f(x) =   r_x ((\bar{f}\circ\iota) (x)) =   r_x f(x).$$
So $\bar{f}$ is uniquely determined by $f$.
\end{proof}

We end this section with a condition of $MC(R)=R$ for a Rota-Baxter algebra $(R,P)$.

\begin{prop}
Let $(R,P)$ be a Rota-Baxter algebra of weight $\lambda$. If $R$ has no zero divisors and $MC(R)=R$, then $P$ is right $R$-linear and either $P(\bfone_{R})=0$ or $P(\bfone_{R})=-\lambda$
\end{prop}
We note that this condition is different from the condition for a Rota-Baxter algebra to be a Rota-Baxter bimodule over itself.

\begin{proof}
From $MC(R)=R$ we have $P(r)=P(\bfone_{R})r$ for all $r\in R$. Then from Eq.~(\ref{eq:rba}) we obtain
$$ P(\bfone_{R})^2=2P^2(\bfone_{R})+\lambda P(\bfone_{R})=2P(\bfone_{R})^2+\lambda P(\bfone_{R}).$$
Thus $P(\bfone_{R})(P(\bfone_{R})+\lambda)=0$ and the conclusion follows.
\end{proof}

\section{Projective and injective resolutions of Rota-Baxter modules}
\mlabel{sec:pi}
We now turn our attention to the Hom functor, the projectivity and the injectivity of Rota-Baxter modules.

\subsection{The Hom functor}
Let ${\bf Ab}$ be the category of abelian groups. Recall that ${\bf _{\rbeta}Mod}$ is the category of left $(R,P)$-modules. If $(M,p_{M})$ and $(N,p_{N})$ are objects of ${\bf _{\rbeta}Mod}$, the set of all the homomorphisms of $\rbeta$-modules from $(M,p_{M})$ to $(N,p_{N})$ will be denoted by $\rhom_{\rbeta}(M,N)$. Thus $\rhom_{\rbeta}(M,N)$ is a subset of $\rhom_R(M,N)$.

\begin{prop}\mlabel{pp:rbab}
Let $(R, P)$ be a Rota-Baxter \bfk-algebra of weight $\lambda$ and $(M,p_M), (N,p_N) \in {\bf _{\rbeta}Mod}$. Then $\rhom_{\rbeta}(M, N)$ is an abelian subgroup of $\rhom_R(M,N)$.
\end{prop}
Thus ${\bf _{\rbeta}Mod}$ is an abelian category.
\begin{proof}
First, the zero element of $\rhom_R(M,N)$ is in $\rhom_{\rbeta}(M,N)$. Next let $f,g\in \rhom_{\rbeta}(M,N)$. The inclusion $\rhom_{\rbeta}(M,N)\subseteq \rhom_R(M,N)$ shows that $f+g$ and $-f$  are $R$-linear.
Further, $f\circ p_M=p_N\circ f$ and $g\circ p_M=p_N\circ g$ give
$$ (f+g) \circ p_M = p_N \circ (f+g), \quad ((-f) \circ p_M)(m)
=( p_N\circ(-f) )(m).
$$
Thus $f+g$ and $-f$ are in $\rhom_{\rbeta}(M, N)$.
Therefore $\rhom_{\rbeta}(M, N)$ is a sub-abelian group of $\rhom_R(M,N)$.
\end{proof}

The following are more generalizations of properties of modules to the category of Rota-Baxter modules. For simplicity, we suppress the adjective left (resp. right and bi-) from a left (resp. right or bi-) Rota-Baxter module when its meaning is clear from the context.
\begin{prop}\label{Hom-module}
Let $(R,P)$, $(S,\alpha)$ and $(T,\gamma)$ be Rota-Baxter algebras.
\begin{enumerate}

\item If $(M_{\rbeta}, p_M^R)$ and $( _{\tgamma}N_{\rbeta}, p_N^T, p_N^R)$ are Rota-Baxter modules, then
 $(\Hom_{\rbeta}(M,N),q)$
is a left $\tgamma$-module with $q$ defined by
$$q(f)(m):=p_N^T(f(m)), \quad f\in \Hom_{\rbeta}(M,N), \quad m\in M.$$ \mlabel{it:bihom1}

\item If $(_{\rbeta} M,p_M^R)$ and $(_{\rbeta} N_{\tgamma}, p_N^R,p_N^T)$ are Rota-Baxter modules, then $(\Hom_{\rbeta}(M,N), q)$
is a right $\tgamma$-module with $q$ defined by
$$q(f)(m): =p_N^T(f(m)), \quad f\in \Hom_{\rbeta}(M, N), \quad m\in M.$$ \mlabel{it:bihom4}

\item If $(_{\rbeta} M_{\salpha},p_M^R, p^S_M)$ and $(_{\rbeta} N,p_N^R)$ are Rota-Baxter modules, then
 $(\Hom_{\rbeta}(M,N),q)$
is a left $\salpha$-module with $q$ defined by
$$q(f)(m):=f(p_M^S(m)), \quad f\in \Hom_{\rbeta}(M,N), \quad m\in M.$$ \mlabel{it:bihom3}

\item  If $(_{\salpha}M_{\rbeta},p_M^S,p_M^R)$ and $(N_{\rbeta},p_N^R)$ are Rota-Baxter modules, then
 $(\Hom_{\rbeta}(M,N),q)$
is a right $\salpha$-module with $q$ defined by
$$q(f)(m):=f(p_M^S(m)), \quad f\in \Hom_{\rbeta}(M, N), \quad m\in M.$$ \mlabel{it:bihom2}
\end{enumerate}
\mlabel{Hom-rbm}
\mlabel{prop:bihom}
\end{prop}

\begin{proof}
(\mref{it:bihom1}).
The $T$-action on $\Hom_R(M,N)$ is defined by
$$(tf)(m):=tf(m) \quad \text{for all } m\in M, f\in \Hom_R(M,N), t\in T.$$
If $f$ is further in $\Hom_{\rbeta}(M,N)$, then $f\circ p^R_M=p^R_N\circ f$. Thus by $N$ being a $(T,\gamma)$-$(R,P)$ bimodule, we obtain
$$(tf)\circ p^R_M(m)=tf(p^R_M(m))=tp^R_N(f(m))=p^R_N(tf(m)) \quad \text{for all } m\in M.$$
Thus $\Hom_{\rbeta}(M, N)$ is a left $T$-submodule of $\Hom_R(M,N)$.

Now we show that $q(f)$ is in $\Hom_{\rbeta}(M,N)$. Since $f$ and $p_N^T$ are right $R$-module homomorphisms, so is their composition $q(f)$. Likewise, since
$f\circ p^R_M=p^R_N\circ f$ from $f\in \Hom_{\rbeta}(M,N)$ and $p^T_N\circ p^R_N=p^R_N\circ p^T_N$ from $N$ being a $(T,\gamma)$-$(R,P)$-bimodule, we have $q(f)\circ p^R_M=p^R_N\circ q(f)$.
Thus $q(f)$ is in $\Hom_{\rbeta}(M,N)$.

We are left to prove
$$\gamma(t)q(f) = q(\gamma(t)f) + q( tq(f)) + \lambda q(tf)\,\text{ for all }t\in T.$$
But this follows from
\begin{align*}
(\gamma(t)q(f))(m) &= \gamma(t)(q(f)(m)) \\
&=\gamma(t)p_N^T(f(m)) \\
&= p_N^T(\gamma(t)f(m)) + p_N^T(tp_N^T(f(m))) + \lambda p_N^T(tf(m)) \quad (\text{as }(N,p^T_N) \text{ is a left } (T,\gamma)\text{-module})\\
&= p_N^T( (\gamma(t)f)(m) ) + p_N^T( t(q(f)(m)) ) + \lambda p_N^T( (tf)(m))\\
&= q(\gamma(t)f) (m) + q( tq(f))(m) + \lambda q(tf)(m).
\end{align*}
The proof of Item~(\mref{it:bihom4}) is similar.
\smallskip

\noindent
(\mref{it:bihom3}). Similar to Item (\mref{it:bihom1}), the $S$-action on $\Hom_{\rbeta}(M,N)$ is defined by
$$(sf)(m)=f(ms), \quad \text{for all } m\in M, s\in S, f\in \Hom_{\rbeta}(M,N).$$
Then it follows in the same way that $q(f)$ is in $\Hom_{\rbeta}(M,N).$
To prove
$$\alpha(s)q(f)=q(sq(f))+q(\alpha(s)f)+\lambda q(sf) \quad \text{for all } s\in S, f\in \Hom_{\rbeta}(M,N),$$
we derive
\begin{align*}
(\alpha(s)q(f))(m) &= q(f)(m\alpha(s))\quad (\text{by the definition of $S$-action})\\
&=f(p_M^S(m\alpha(s)))\quad (\text{by the definition of $q(f)$})\\
&= f(p_{M}^{S}(p_{M}^{S}(m)s)+p_{M}^{S}(m)\alpha(s)+\lambda (p_{M}^{S}(m)s))\\
&\quad (\text{by the definition of Rota-Baxter right module })\\
&= f(p_{M}^{S}(p_{M}^{S}(m)s))+f(p_{M}^{S}(m)\alpha(s))+\lambda f(p_{M}^{S}(m)s)\\
&= q(f)(p_{M}^{S}(m)s)+(\alpha(s)f)(p_{M}^{S}(m))+\lambda (sf)(p_{M}^{S}(m)))\\
&= (sq(f))(p_{M}^{S}(m))+(\alpha(s)f)(p_{M}^{S}(m))+\lambda (sf)(p_{M}^{S}(m))\\
&= q(sq(f))(m)+q(\alpha(s)f)(m)+\lambda q(sf)(m)\\
&= (q(sq(f))+q(\alpha(s)f)+\lambda q(sf))(m)\\
&\quad (\text{by the definition of $S$-action and $q(f)$}),
\end{align*}
as required. The proof of Item (\mref{it:bihom2}) is similar.
\end{proof}

\subsection{Projective and injective Rota-Baxter modules}
By Proposition~\mref{pp:rbab}, the category of left $(R,P)$-modules is an abelian category. By~\cite[\S~2.5]{We}, for an abelian category with enough project and injective objects, one can define derived functors of $\Hom$ using projective resolutions and injective resolutions.
Thus we just need to prove that there are enough projective and injective left $(R,P)$-modules.

We first give the definition of projective left Rota-Baxter modules.

\begin{defn}
Let $(R,P)$ be a Rota-Baxter \bfk-algebra of weight $\lambda$. A left $(R,P)$-module $(V, p)$ is
{\bf projective} if, for every left $(R,P)$-module epimorphism $f:(N,p_N)\to (M,p_M)$ and every left $(R,P)$-module homomorphism $g:(V,p)\to (M,p_M)$,
there exists a left $\rbeta$-module homomorphism $\bar{g}:(V,p)\to (N,p_N)$ making the following diagram commutative:
$$
\xymatrix{
&(V,p)\ar^{g}[d] \ar@{-->}_{\bar{g}}[dl]&\\
(N,p_{N})\ar_{f}[r]&(M,p_{M})\ar[r]&0.}
$$
\mlabel{defn:proj}
\end{defn}

\begin{prop}
A free left Rota-Baxter module is a projective left Rota-Baxter module.
\mlabel{prop:frep}
\end{prop}

\begin{proof}
The proof is the same as the case for left modules. We give some details for completeness. Let $(F(X),p)$ be the free left Rota-Baxter $(R,P)$-module on $X$ with the natural embedding $j_X: X\longrightarrow F(X)$.
Let $f:(N,p_N)\to (M, p_M)$ be a surjective $(R,P)$-module homomorphism and let $$g: (F(X),p) \to (M,p_X)$$ be a left Rota-Baxter module homomorphism.
Since $f$ is surjective, for each $x\in X$,
there is a $n_x\in N$ such that $f(n_x) = g(x)$. Define a map $g_0: X \to N$ by $x\mapsto n_x$.
Then by the universal property of $F(X)$,
there is a left $(R, P)$-module homomorphism $\bar{g}:F(X) \to N$ such that $ \bar{g}\circ j_X = g_0$. So  $ f\circ \bar{g}\circ j_X = f\circ g_0$.
Again by the universal property of $F(X)$, we have  $ f\circ \bar{g}= g$. This is what we need.
\end{proof}

From Corollary~\mref{coro:quof} and Proposition~\mref{prop:frep}, we obtain that there are enough projective objects in the category of Rota-Baxter modules.

We next introduce the concept of an injective Rota-Baxter module and show that there are enough injective objects in the category of left Rota-Baxter modules, namely every left Rota-Baxter module can be embedded into an injective left Rota-Baxter module. We take a similar approach as in the case of modules, but the process becomes more involved.

\begin{defn}
Let $(R,P)$ be a Rota-Baxter \bfk-algebra of weight $\lambda$.
A left $(R,P)$-module $(E,p)$ is {\bf injective} if, whenever $f$ is a left $\rbeta$-module monomorphism
and $g$ is a left $(R,P)$-module homomorphism,
there exists a left $\rbeta$-module homomorphism $\bar{g}$ making the following diagram commutative:
$$
\xymatrix{
&(E,p)\ar@{<--}^{\bar{g}}[dr]&\\
0\ar[r]&(N,p_{N})\ar_{f}[r]\ar^{g}[u]&(M,p_{M}).}
$$
\end{defn}

We first recall the concept and construction of the ring of Rota-Baxter operators given in~\mcite{GL}.

\begin{defn}\label{rbor}
Let $(R,P)$ be a Rota-Baxter algebra of weight $\lambda$ and $\bfk\langle R, Q \rangle$ be the free product of the \bfk-algebras $R$ and $\bfk[Q]$, where $Q$ is a variable.
The {\bf ring of Rota-Baxter operators on} $(R,P)$, denoted by $R_{RB}\langle Q \rangle$, is defined to be the quotient
\begin{equation*}
R_{RB}\langle Q \rangle=\bfk\langle R, Q \rangle/I_{R,Q},
\mlabel{de:RBO}
\end{equation*}
where $I_{R,Q}$ is the ideal of $\bfk\langle R, Q \rangle$ generated by the subset
$$
\{ QrQ - P(r)Q + QP(r)+\lambda Qr~|~r\in R\}.
$$
Let $\bfone_{R_{RB}\langle Q \rangle}$ denote the identity of $R_{RB}\langle Q \rangle$.
\end{defn}

There is the following correspondence between Rota-Baxter modules and $R_{RB}\langle Q\rangle$-modules~\cite{GL}:

\begin{prop}
If $(M,p)$ is a left $(R,P)$-module, then the resulting left $R$-module $M$ together with $Q\cdot m:=p(m), m\in M$, makes $M$ into a left $R_{RB}\langle Q\rangle$-module. Conversely, if $M$ is a left $R_{RB}\langle Q\rangle$-module, then $(M,p)$ is a left $(R,P)$-module, where $p:M\to M, p(m):=Qm, m\in M$. In particular, left $(R,P)$-ideals of $R_{RB}\langle Q\rangle$ are of the form $(S,\bar{P}|_S)$ where $S$ is a left ideal of $R_{RB}\langle Q\rangle$ and $\bar{P}: R_{RB}\langle Q\rangle \to R_{RB}\langle Q\rangle$ is the left multiplication by $Q$.
\label{pp:corr}
\end{prop}

Applying this result, we next give the Rota-Baxter module version of the Baer Criterion for injectivity of modules.

\begin{prop}\label{RB-bare-cri}
Let $(V,p)$ be a left $(R,P)$-module. Then $(V,p)$ is an injective left $(R,P)$-module if and only, for every left $(R,P)$-ideal $(S,\bar{P}|_S)$ of $(R_{RB}\langle Q \rangle,\bar{P})$, every $(R,P)$-module homomorphism $f: (S,\overline{P}|_{S})\longrightarrow (V,p)$ can be extended to one from $(R_{RB}\langle Q \rangle,\overline{P})$.
\end{prop}
\begin{proof}
We adapt the proof of the Baer Criterion as presented for example in~\mcite{Rot}.

Suppose that $(V,p)$ is an injective $(R,P)$-module. Then by the definition of an injective Rota-Baxter module, every $(R,P)$-module homomorphism $f: (S,\overline{P}|_{S})\longrightarrow (V,p)$ can be extended to one from $(R_{RB}\langle Q \rangle,\overline{P})$.

Conversely, assume that, for every left $(R,P)$ ideal $(S,\bar{P}|_S)$ of $R_{RB}\langle Q \rangle$, every $(R,P)$-module homomorphism $f: (S,\overline{P}|_{S})\longrightarrow (V,p)$ can be extended to one from $(R_{RB}\langle Q \rangle,\overline{P})$.

Let $f: (N,p_N)\to (M,p_M)$ be a monomorphism of left $(R,P)$-modules and let $g: (N,p_N)\to (V,p)$ be a left $(R,P)$-module homomorphism. Identify $(N,p_N)$ as a left $(R,P)$-submodule of $(M,p_M)$ and denote
$$\mathcal{S}:=\{ h:(H,p_H)\to (V,p)\,|\, (N,p_N)\leq (H,p_H)\leq  (M,p_M),  h|_{(N,p_N)}=g\}.$$
Then $\mathcal{S}$ is non-empty since it contains $g:(N,p_N)\to (V,P)$. Define a partial order on $\mathcal{S}$ by the inclusion of the domains $(H,p_H)$. Then $\mathcal{S}$ contains a maximal element $h:(H,p_H)\to (V,p)$ by Zorn'lemma. If $H=M$, then we are done. Supposing not, then take $b\in M\backslash H$. Regard $H$ as a left $R_{RB}\langle Q\rangle$-module by Proposition~\mref{pp:corr} and denote
$$ L:=\{ r\in R_{RB}\langle Q\rangle \,|\, rb \in H \},$$
which is a left $R_{RB}\langle Q\rangle$-module. Then $(L,\bar{P}|_{L})$ is a left $(R,P)$-module and the composition
$$\eta: L\to H \to V, \quad r\mapsto rb \mapsto h(rb)$$
is a well-defined $(R,P)$-module homomorphism. By the assumption, there is an $(R,P)$-module homomorphism $\varphi:(R_{RB}\langle Q\rangle,\bar{P})\to (V,p)$ such that $\varphi(r)=h(rb)$ for $r\in L$. Denote $c:=\varphi(\bfone_{R_{RB}\langle Q\rangle})$ and define a map
$$ \psi: H+R_{RB}\langle Q\rangle b \longrightarrow V, \quad
a+rb \mapsto h(a)+rc, \quad a\in H, r\in R_{RB}\langle Q\rangle.$$
If $a+rb=a'+r'b$ with $a, a'\in H$ and $r,r'\in R_{RB}\langle Q\rangle$, then
$$h(a-a')=h((r'-r)b)=\varphi(r'-r)=(r'-r)c\,\text{ and so }\, h(a)+rc=h(a')+r'c,$$
which implies that $\psi$ is well-defined.
Then $\psi$ is a left $R_{RB}\langle Q\rangle$-module homomorphism and hence, by Proposition~\ref{pp:corr}, a left $(R,P)$-module homomorphism extending $g$. Hence it is in $\mathcal{S}$ and is strictly larger than $h$. This is a contradiction. Thus we must have $H=M$.
\end{proof}

Recall that an abelian group $G$ is called a divisible abelian group, if for any $x\in G$ and any nonzero integer $n\in \mathbb{Z}$, there is some $y\in G$ such that $x=ny$.

\begin{prop}\mlabel{injective-rbm}
Let $(R,P)$ be a Rota-Baxter \bfk-algebra of weight $\lambda$ and
$D$ be a divisible abelian group. Then $(\Hom_{\mathbb{Z}}(R_{RB}\langle Q \rangle,D),q)$ is an injective left $(R,P)$-module.
\end{prop}
\begin{proof}
Applying Proposition~\mref{RB-bare-cri}, we let $S$ be a left ideal of $R_{RB}\langle Q \rangle$ and let $\eta: (S,{\bar{P}}|_{S})\longrightarrow (R_{RB}\langle Q \rangle,\bar{P})$ be the embedding map. For any
$$f:(S,\bar{P}|_S)\to (\Hom_{\mathbb{Z}}(R_{RB}\langle Q \rangle,D),q),$$
we extend $f$ as in the following diagram:
$$
\xymatrix{
&(\Hom_{\mathbb{Z}}(R_{RB}\langle Q \rangle,D), q)\ar@{<--}^{g}[dr]&\\
0\ar[r]&(S, \bar{P}|_{S})\ar_{\eta}[r]\ar^{f}[u]&(R_{RB}\langle Q \rangle, \bar{P}).}
$$

Define $\phi: S\longrightarrow D$ by $\phi(s)=f(s)(\bfone_{R_{RB}\langle Q \rangle})\in D$. Then $\phi$ is a $\mathbb{Z}$-module homomorphism. Since an abelian group is an injective $\mathbb{Z}$-module if and only if it is a divisible abelian group~\mcite{Rot}, $D$ is an injective $\mathbb{Z}$-module. Then there is a $\mathbb{Z}$-module homomorphism $\psi: R_{RB}\langle Q \rangle\longrightarrow D$ such that $\phi=\psi \circ \eta$.

For any $x, y\in R_{RB}\langle Q \rangle$, define $g(x)(y)=\psi(yx)$. Then $g$ is a map from $R_{RB}\langle Q \rangle$ to $\Hom_{\mathbb{Z}}(R_{RB}\langle Q \rangle,D)$. Let $r\in R$. Then $$g(rx)(y)=\psi(y(rx))=\psi((yr)x)=g(x)(yr)=(rg(x))(y)$$
and so $g$ is an $R$-module homomorphism. Since
\begin{eqnarray*}
((g\circ\bar{P})(x))(y)=(g(\bar{P}(x)))(y)=\psi(y\bar{P}(x))
=\psi(y(Q x))=\psi((y Q)x))
=g(x)(y Q)=((q\circ g)(x))(y),
\end{eqnarray*}
$g: (R_{RB}\langle Q\rangle,\bar{P})\longrightarrow (\Hom_{\mathbb{Z}}(R_{RB}\langle Q\rangle,D),q)$ is an $(R,P)$-module homomorphism. Let $s\in S$. For $x\in R$, we have
$$
((g\circ\eta)(s))(x)=g(s)(x)=\psi(xs)=\phi(xs)=f(xs)(\bfone_{R_{RB}\langle Q \rangle})=(xf(s))(\bfone_{R_{RB}\langle Q \rangle})=f(s)(x).
$$
For $x=Q$, we have
\begin{eqnarray*}
((g\circ\eta)(s))(Q)&=&g(s)(Q)=\psi(Qs)=\phi(Qs)=f(Qs)(\bfone_{R_{RB}\langle Q \rangle})\\
&=&((f\circ\bar{P})(s))(\bfone_{R_{RB}\langle Q \rangle})=((q\circ f)(s))(\bfone_{R_{RB}\langle Q \rangle})\\
&=&(q(f(s)))(\bfone_{R_{RB}\langle Q \rangle})=f(s)(Q),
\end{eqnarray*}
which implies $g\circ \eta=f$. Hence $(\Hom_{\mathbb{Z}}(R_{RB}\langle Q \rangle,D),q)$ is an injective $(R,P)$-module by Proposition~\ref{RB-bare-cri}.
\end{proof}

\begin{theorem}
Let $(R,P)$ be a Rota-Baxter \bfk-algebra of weight $\lambda$ and $(V,p)$ be a left $(R,P)$-module. Then $(V,p)$ can be embedded into an injective $(R,P)$-module.
\mlabel{thm:sinj}
\end{theorem}
\begin{proof}
Define
\begin{equation*}
R_{RB}\langle Q \rangle\times V\longrightarrow V, \ (r,m)\longmapsto rm, (Q,m) \longmapsto p(m), \ r\in R, m\in V.
\mlabel{eq:rbm-Q}
\end{equation*}
Then V is an $R_{RB}\langle Q \rangle$-module. Now define
$$f: (V,p)\longrightarrow (\Hom_{\mathbb{Z}}(R_{RB}\langle Q \rangle,V),q), \quad m\mapsto \varphi_{m},$$
where $\varphi_{m}(x)=xm$ for $x\in R_{RB}\langle Q \rangle$. Thus $\varphi_{m}$ is a $\mathbb{Z}$-module homomorphism. For any $r\in R$, $x\in R_{RB}\langle Q \rangle$ and $m\in V$, we have
$$
f(rm)(x)=x(rm)=(xr)m=\varphi_{m}(xr)=r\varphi_{m}(x)=(rf(m))(x),
$$
and so $f$ is an $R$-module homomorphism. Since
\begin{eqnarray*}
((f\circ p)(m))(x)&=&f(p(m))(x)=f(Qm)(x)=x(Qm)=(xQ)m\\
&=&\varphi_{m}(xQ)=f(m)(xQ)=(q(f(m)))(x)=(q\circ f(m))(x),
\end{eqnarray*}
$f$ is an $(R,P)$-module homomorphism. We now show that it is a monomorphism. For any $m, m'\in V$, if $\varphi_{m}=\varphi_{m'}$, then $xm=\varphi_{m}(x)=\varphi_{m'}(x)=xm'$ for all $x\in R_{RB}\langle Q \rangle$. In particular, this is true for $x = \bfone_{R_{RB}\langle Q \rangle}$, and so $m =m'$.

Since every abelian group can be embedded into a divisible abelian group~\mcite{Rot}, there exists an embedding map $\eta_{1}: V \longrightarrow D$.
Now define $\bar{\eta}: (\Hom_{\mathbb{Z}}(R_{RB}\langle Q \rangle,V), q)\longrightarrow (\Hom_{\mathbb{Z}}(R_{RB}\langle Q \rangle,D), q')$ by $\tau\longmapsto \eta_{1}\circ\tau$. For any $r\in R$ and $x\in R_{RB}\langle Q \rangle$, we have
$$
(\bar{\eta}(r\tau))(x)=\eta_{1}((r\tau)(x))=(\eta_{1}\circ\tau)(xr)=r(\eta_{1}\circ\tau)(x)
=(r\bar{\eta}(\tau))(x),
$$
and so $\bar{\eta}$ is an $R$-module homomorphism. Moreover, since
\begin{align*}
((q'\circ\bar{\eta})(\tau))(x)&=(q'(\bar{\eta}(\tau)))(x)=\bar{\eta}(\tau)(xQ)= (\eta_{1}\circ\tau)(xQ)=(\eta_{1} (q(\tau)))(x)=((\bar{\eta}\circ q)(\tau))(x),
\end{align*}
$\bar{\eta}$ is an $(R,P)$-module homomorphism. Therefore, $(\Hom_{\mathbb{Z}}(R_{RB}\langle Q \rangle,D),q')$ is an injective $(R,P)$-module by Proposition~\mref{injective-rbm} and
$$
\xymatrix{
\bar{\eta}\circ f: (V, p) \ar^{f \qquad }[r] & (\Hom_{\mathbb{Z}}(R_{RB}\langle Q \rangle,V),q) \ar^{\bar{\eta}}[r]& (\Hom_{\mathbb{Z}}(R_{RB}\langle Q \rangle,D),q'),
}
$$
 is an $(R,P)$-monomorphism, as required.
\end{proof}

\section{Flat Rota-Baxter modules}
\mlabel{sec:flat}

We finally turn to the study of flat Rota-Baxter modules, beginning with the construction of the tensor product of two Rota-Baxter modules in the category of Rota-Baxter modules.

\subsection{Tensor product of Rota-Baxter modules}
We first define the tensor product of Rota-Baxter modules.

\begin{defn}
Let $(R,P)$ be a Rota-Baxter algebra of weight $\lambda$, $(M_{\rbeta},p_{M})$
 a right $(R,P)$-module and $(_{\rbeta}N,p_{N})$
a left $(R,P)$-module.
\begin{enumerate}
\item
Let $G$ be an (additive) abelian group. A map $f: M\times N\longrightarrow G$ is called $(R,P)$-{\bf bilinear} if for all $m, m'\in M$, $n, n'\in N$ and $r\in R$, we have
\begin{align*}
f(m+m', n)&=f(m, n)+f(m', n),\\
f(m, n+n')&=f(m, n)+f(m, n'),\\
f(mr, n)&=f(m, rn),\\
f(p_{M}(m), n)&=f(m, p_{N}(n)).
\end{align*}
\item
The $\bf tensor~product$ $M\otimes_{(R,P)} N$ of  $(M_{\rbeta},p_{M})$ and $(_{\rbeta}N,p_{N})$ over $(R,P)$ is an abelian group together with
a $(R,P)$-bilinear map
$$ \iota: M\times N\longrightarrow M\otimes_{(R,P)} N$$
satisfying the following universal property: for every abelian group $G$ and every $(R,P)$-bilinear map $f: M\times N\longrightarrow G$,
there exists a unique abelian group homomorphism $\widetilde{f}: M\otimes_{\rbeta} N\longrightarrow G$ making the following diagram commutative
$$
\xymatrix{
&M\times N\ar^{\iota}[rr] \ar_{f}[dr]&&M\otimes_{\rbeta} N.\ar@{-->}^{\widetilde{f}}[dl]\\
&&G&}
$$
\end{enumerate}
\mlabel{def:rbmt}
\end{defn}

The following result gives a construction of the tensor product of Rota-Baxter modules.

\begin{theorem}
Let $(R,P)$ be a Rota-Baxter algebra of weight $\lambda$ and let $(M_{(R,P)},p_{M})$, $(_{(R,P)}N,p_{N})$ be Rota-Baxter modules.
Let $F$ be the free abelian group on the set $M \times N$ and $I$ the subgroup of $F$ generated by all elements of $F$
of the form
\begin{align*}
&(m+m', n) - (m, n)-(m', n),\ (m, n+n')- (m, n) - (m, n'),\\
&(mr, n)- (m, rn),\  (p_{M}(m), n)-(m, p_{N}(n)),\  m, m'\in M, n, n'\in N,r\in R.
\end{align*}
Then $F/I$ with the natural map $\iota: M \times N \rightarrow F \rightarrow F/I$ is the tensor product $M\otimes_{(R,P)} N$.
\mlabel{thm:rbmt}
\end{theorem}

\begin{proof}
For $(m,n)\in M\times N$, write $m\ot_{\rbeta} n:= \iota((m,n))$, called a pure tensor. Then elements in $F/I$ are finite sums of pure tensors. We verify the desired universal property of $F/I$.
\smallskip

Let $f: M\times N\longrightarrow G$ be a $(R,P)$-bilinear map. Then
$f$ extends to an abelian group homomorphism $f': F \longrightarrow G$ by additivity. Since
$f'$ vanishes on the generators of $I$, $f'$ induces a well-defined abelian group homomorphism
$\tilde{f}: F/I \longrightarrow G$ such that $f'((m,n)) = \tilde{f}(m\ot_{\rbeta}n)$ with $m\in M$ and $n\in N$.
So we have
$$ f(m,n) = f'(m,n) = \tilde{f}(m\ot_{\rbeta}n) = \tilde{f}\circ \iota(m,n),$$
as required.
\smallskip

If $\widetilde{f}$ satisfies the conditions, then
$$
\widetilde{f}(\sum_{i} m_i \otimes_{\rbeta} n_i) =  \sum_i \widetilde{f}( m_i \otimes_{\rbeta} n_i)
=  \sum_i \widetilde{f}(\iota( m_i, n_i))
= \sum_i f(m_i, n_i).
$$
So $\widetilde{f}$ is uniquely determined by $f$.
\end{proof}

\begin{prop}
Let $(R,P)$ be a Rota-Baxter algebra of weight $\lambda$.
\begin{enumerate}
\item If $(M_{\rbeta},p_M)$ is a right $\rbeta$-module,
there is an additive functor $F_{M}: {\bf _{\rbeta}Mod} \longrightarrow {\bf Ab}$
defined by
$$F_{M}(N) = M \otr N , \quad F_M(g) = id_M \otr g,$$
where $(N,p_N), (L,p_{L}) \in {\bf _{\rbeta}Mod}$ and $g:(N,p_N) \longrightarrow (L,p_{L})$ is a left $(R,P)$-module homomorphism.
\mlabel{it:tc1}

\item If $( _{\rbeta}M,p_M)$ is a left Rota-Baxter module,
there is an additive functor $G_{M}: {\bf Mod_{\rbeta}} \longrightarrow {\bf Ab}$
defined by
$$G_{M}(N) = N \otr M, \quad G_M(g) = g \otr id_M,$$
where $(N,p_N), (L,p_{L}) \in {\bf Mod_{\rbeta}}$ and $g:(N,p_N) \longrightarrow (L,p_{L})$ is a right $(R, P)$-module homomorphism.
\mlabel{it:tc2}
\end{enumerate}
\mlabel{prop:tc}
\end{prop}

\begin{proof}
(\mref{it:tc1}) Let $g':(L,p_L) \longrightarrow (H,p_{H})$ be a left $(R, P)$-module homomorphism with $(L,p_L), (H,p_{H}) \in {\bf Mod_{\rbeta}}$. Then
$$F_M(g\circ g') = id_M \otr (g\circ g') = (id_M \otr g)\circ(id_M \otr g') = F_M(g)\circ F_M(g').  $$
Since $F_M(id_N) = id_M \otr id_N$, $F_M$ is a functor. We are left to show $$F_M(g+h) = F_M(g) + F_M(h),$$
where $g, h: N \longrightarrow L$ are left $(R,P)$-module homomorphism. Let $m\otr n\in M\otr N$. Then
\begin{align*}
F_M(g+h)(m\otr n) &= m \otr ((g+h)(n)) \\
&= m \otr (g(n) + h(n)) \\
&= m\otr g(n) + m \otr h(n) \\
&= ( F_M(g) + F_M(h)) (m\otr n),
\end{align*}
as required.

(\mref{it:tc2}) The proof is similar to Item~(\mref{it:tc1}).
\end{proof}

\begin{prop}({\bf Extension of scalars})\label{TP-m}
Let $(R,P)$ and $(S,\alpha)$ be Rota-Baxter algebras of weight $\lambda$.
\begin{enumerate}
\item If $(_{(S,\alpha)} M_{\rbeta}, p_M^S, p_M^R)$ is a Rota-Baxter bimodule and $(_{\rbeta}N, p_N^R)$ is a left $\rbeta$-module,
then $(M \otimes_{\rbeta}N,q)$ is a left $(S,\alpha)$-module by defining
\begin{align*}
s(m\ot_{\rbeta}n):&= (sm) \ot_{\rbeta} n,\\
q(m\otr n):&= p_M^S(m) \otr n, \quad {\text where }\ s\in S, m\in M, n\in N.
\end{align*}
\mlabel{it:es1}
\item If $(M_{\rbeta}, p_M^R)$ is a right $\rbeta$-module and $(_{\rbeta}N_{(S,\alpha)}, p_N^R, p_N^S)$ is a  Rota-Baxter bimodule,
then $(M \otimes_{\rbeta}N,q)$ is a right $(S,\alpha)$-module by defining
\begin{align*}
(m\ot_{\rbeta}n)s:&= m \ot_{\rbeta} (ns),\\
q(m\otr n):&= m \otr p_N^S(n), \quad {\text where }\ s\in S, m\in M, n\in N.
\end{align*}
\mlabel{it:es2}
\end{enumerate}
\mlabel{prop:es}
\end{prop}

\begin{proof}
(\mref{it:es1}) It is straightforward to check that $M \otimes_{\rbeta}N$ is a left $S$-module.
So we are left to verify Eq.~(\mref{eq:lrbm}). Let $s\in S, m\in M, n\in N$. Then
\begin{align*}
\alpha(s) q(m\ot_{\rbeta} n ) &=  \alpha(s)( p_{M}^S(m) \otr n)\\
&= (\alpha(s) p_M^S(m)) \otr n\\
&= p_M^S(\alpha(s)m) \otr n + p_M^S( sp_M^S(m)) \otr n + \lambda p_M^S(sm) \otr n\\
&= q((\alpha(s)m) \otr n) + q(sp_M^S(m) \otr n) + \lambda q(sm \otr n)\\
&= q(\alpha(s) (m\otr n)) + q(s(p_M^S(m) \otr n)) + \lambda q(s(m\otr n))\\
&= q(\alpha(s) (m\otr n)) + q(s(q(m \otr n)) ) + \lambda q(s(m\otr n)),
\end{align*}
as required.

(\mref{it:es2}) The proof is similar to Item~(\mref{it:es1}).
\end{proof}

The next result shows that $\square \otr N$ and $\Hom_{\salpha}(N,\square)$ are adjoint functors.

\begin{theorem}
Let $(R,P)$ and $(S,\alpha)$ be Rota-Baxter algebras of weight $\lambda$.
Let $(M_{\rbeta},p_M^R)$ be a right $\rbeta$-module, $( _{\rbeta}N_{\salpha},p_N^R, p_N^S)$ a
Rota-Baxter bimodule and $(L_{\salpha},p_L^S)$ a right $\salpha$-module. Then

$$ \Hom_{\salpha}(M\otr N,L) \cong \Hom_{\rbeta}(M,\Hom_{\salpha}(N,L)).$$
\end{theorem}

\begin{proof}
Define
\begin{align*}
\tau:  \Hom_{\salpha}(M\otr N,L) &\longrightarrow \Hom_{\rbeta}(M,\Hom_{\salpha}(N,L)),\\
f &\longmapsto \tau(f), \text{ where } \tau(f)(m): n\longmapsto f(m\otr n).
\end{align*}
Then $\tau$ is the required isomorphism.
\end{proof}

\subsection{Flat Rota-Baxter modules}
As in the classical case, it is quite routine to check that the Rota-Baxter tensor product is right exact. To study the exactness of the tensor product, we introduce the flatness condition in the context of Rota-Baxter modules.

\begin{defn}
Let $\rbeta$ be an Rota-Baxter algebra of weight $\lambda$. A right $\rbeta$-module $(M,p)$ is $\bf flat$ if $M\otimes _{\rbeta} \Box$ is an exact functor, that is, whenever
$$
\xymatrix{
0\ar[r]&(N',p_{N'})\ar^{i}[r]&(N,p_N)\ar^{j}[r]&(N'',p_{N''})\ar[r]& 0
}
$$
is an exact sequence of left $\rbeta$-modules, then
$$
\xymatrix{
0\ar[r]& M\otimes_{\rbeta}N'\ar^{id_{M}\otimes i}[r]& M\otimes_{\rbeta} N\ar^{id_{M}\otimes j~~}[r]& M\otimes_{\rbeta} N''\ar[r]& 0&}
$$
is an exact sequence of abelian groups.
\end{defn}

Since the functors $M\otimes _{\rbeta} \Box$ are right exact, we see that a right $\rbeta$-module $(M,p_{M})$ is flat if and only if, whenever $i: (N',p_{N'})\longrightarrow (N,p_{N'})$ is an injection, then $id_{M}\otimes i: (M\otimes_{\rbeta}N',p')\longrightarrow (M\otimes_{\rbeta}N,p)$ is also an injection.

\begin{theorem}\label{flat-iso}
Let $\rbeta$ be a Rota-Baxter algebra of weight $\lambda$ and $(M,p)$ a right $\rbeta$-module. Suppose the inclusion $R\to R_{RB}\langle Q\rangle$ gives an injective $(R,P)$-module homomorphism $\eta: (R,P)\longrightarrow (R_{RB}\langle Q \rangle,\bar{P})$. If $(M,p)$ is a flat $(R,P)$-module, then $M\otimes_{(R,P)}R\cong M$ as right $R$-modules.
\end{theorem}
\begin{proof}
Since $(M,p)$ is flat, the abelian group homomorphism $id_{M}\otimes_{(R,P)} \eta: M\otimes_{(R,P)}R\longrightarrow M\otimes_{(R,P)}R_{RB}\langle Q \rangle$ is injective.
Let $m\otimes_{R_{RB}\langle Q \rangle}x$ be a pure tensor in $M\otimes_{R_{RB}\langle Q \rangle}R_{RB}\langle Q \rangle$. Then $$(id_{M}\otimes_{(R,P)} \eta)(mx\otimes_{(R,P)}\bfone_{R})=mx\otimes_{(R,P)}\bfone_{R_{RB}\langle Q \rangle}=m\otimes_{R_{RB}\langle Q \rangle}x.$$
Thus $id_{M}\otimes_{(R,P)} \eta$ is surjective and so is an abelian group isomorphism.
By the extension of scalars in Proposition~\mref{prop:es}, $M\otimes_{(R,P)}R$ and $M\otimes_{(R,P)}R_{RB}\langle Q \rangle$ are right $R$-modules. For any $m\otimes r\in M\otimes_{(R, P)} R$ and $r'\in R$, we have
$$ (id_M \otimes \eta) ((m\otimes r')r) = (id_M \otimes \eta) (m\otimes r' r) = m \otimes \eta(r' r) = m \otimes \eta(r') r
= ((id_M \otimes \eta) (m\otimes r')) r $$
and so $id_{M}\otimes_{(R,P)} \eta$ is an isomorphism of right $R$-modules.
Furthermore regard $(M,p)$ and $(R_{RB}\langle Q \rangle,\bar{P})$ as $R_{RB}\langle Q \rangle$-modules by Proposition~\ref{pp:corr}, we have $M\otimes_{(R,P)}R_{RB}\langle Q \rangle=M\otimes_{R_{RB}\langle Q \rangle}R_{RB}\langle Q \rangle\cong M$ as right $R_{RB}\langle Q \rangle$-modules and also
 as $R$-modules.
Hence $M\otimes_{(R,P)}R\cong M\otimes_{(R,P)}R_{RB}\langle Q \rangle\cong M$ as right $R$-modules.
\end{proof}

Now we give an example of a Rota-Baxter algebra satisfying the conditions in Theorem~\mref{flat-iso}.

\begin{exam}
Let $\rbeta$ be a Rota-Baxter algebra of weight $\lambda$ with $P(r)=-\lambda r$ as in Proposition \mref{pp:rbbm}. Then by Definition \mref{rbor} we have
$$QrQ - P(r)Q +QP(r) +\lambda Qr = QrQ +\lambda rQ - \lambda Qr +\lambda Qr = QrQ +\lambda rQ= (Q+\lambda)rQ = 0$$
and so $Q = -\lambda $ in $R_{RB}\langle Q\rangle$. Hence for the $\eta$ in Theorem~\mref{flat-iso},
we get
$$(\eta\circ P)(r)=\eta(P(r))= \eta(-\lambda r) = -\lambda r = Qr = Q\eta(r) = (\bar{P}\circ \eta)(r)$$
for $r\in R$ and so $\eta$ is an injective $(R,P)$-module homomorphism.
\end{exam}

Let $\{(M_{i},p_{i})~|~i \in I \}$ be a family of left $\rbeta$-modules. Then $\Big(\bigoplus_{i\in I}M_{i},\bigoplus_{i\in I}p_{i}\Big),$
where $\bigoplus_{i \in I} p_{i}$ is defined by
$$
\Big(\bigoplus_{i\in I}p_{i}\Big)(m_{i})_{I} = (p_{i}(m_{i}))_{I},
$$
is also a left $\rbeta$-module and is called the direct sum of $\{(M_{i},p_{i})~|~i \in I \}$.
It is easy to see
$$
\bigoplus_{i\in I}(M_{i},p_{i})=\Big(\bigoplus_{i\in I}M_{i},\bigoplus_{i\in I}p_{i}\Big).
$$
For each $i\in I$, the map $\iota_{i}: (M_{i},p_{i})\longrightarrow \bigoplus_{i\in I}(M_{i},p_{i})$ is a monomorphism and satisfies $(\bigoplus_{i\in I} p_{i})\circ \iota_{i}=\iota_{i}\circ p_{i}$. The map $\rho_{i}: \bigoplus_{i\in I}(M_{i}, p_{i})\longrightarrow (M_{i},p_{i})$ is an epimorphism and satisfies $p_{i}\circ \rho_{i}=\rho_{i}\circ (\bigoplus_{i\in I} p_{i})$. Further $\iota_{i}\circ \rho_{i}=id_{\oplus_{i\in I}(M_{i},p_{i})}$, and $\rho_{i}\circ \iota_{i}=id_{(M_{i}, p_{i})}$.

\begin{lemma}\label{dir-sum-injective}
Let $\rbeta$ be a Rota-Baxter {\bf k}-algebra of weight $\lambda$, and $\{(M_{i},p_{i})~|~i \in I \}$, $\{(N_{i},q_{i})~|~i \in I \}$ be two families of $\rbeta$-modules. Let $\varphi_i: (M_{i},p_{i})\longrightarrow (N_{i},q_{i})$ be $(R,P)$-module homomorphisms.
Then the $\rbeta$-module homomorphism
$$\varphi:= \bigoplus_{i\in I}\varphi_i: \bigoplus_{i\in I}(M_{i},p_{i})\longrightarrow \bigoplus_{i\in I}(N_{i},q_{i}), \quad
(m_{i})_{I} \longmapsto (\varphi_{i}(m_{i}))_{I},
$$
is injective if and only if each $\rbeta$-module homomorphism $\varphi_{i}: (M_{i},p_{i})\longrightarrow (N_{i},q_{i})$ is injective.
\end{lemma}
\begin{proof}
This follows from $\ker \varphi = \oplus_{i\in I} \ker \varphi_i$.
\end{proof}

\begin{lemma}\label{dir-otim-iso}
Let $\rbeta$ be a Rota-Baxter {\bf k}-algebra of weight $\lambda$, $\{(M_{i},p_{i})~|~i \in I \}$ be a family of left $\rbeta$-modules, and $(L, p)$ be a right $\rbeta$-module. Then $L\otimes_{\rbeta}(\oplus_{i\in I}M_{i})\cong \oplus_{i\in I}(L\otimes_{\rbeta}M_{i})$.
\end{lemma}
\begin{proof}
Define group homomorphisms
\begin{eqnarray*}
f: L\otimes_{\rbeta}(\oplus_{i\in I}M_{i}) \longrightarrow \oplus_{i\in I}(L\otimes_{\rbeta}M_{i}), \
\ell\otimes(m_{i})_{I} \longmapsto (\ell \otimes m_{i})_{I},
\end{eqnarray*}
and
\begin{eqnarray*}
g: \oplus_{i\in I}(L\otimes_{\rbeta}M_{i}) \longrightarrow L\otimes_{\rbeta}(\oplus_{i\in I}M_{i}), \
(\ell_{i}\otimes m_{i})_{I} \longmapsto (\prod_{i\in I} \ell_{i}) \otimes (m_{i})_{I}.
\end{eqnarray*}
It is easy to check that $f\circ g=id_{\oplus_{i\in I}(L\otimes_{\rbeta}M_{i})}$ and $g\circ f=id_{L\otimes_{\rbeta}(\oplus_{i\in I}M_{i})}$. Then $L\otimes_{\rbeta}(\oplus_{i\in I}M_{i})\cong \oplus_{i\in I}(L\otimes_{\rbeta}M_{i})$.
\end{proof}

\begin{prop}\label{dir-sum-flat}
Let $\rbeta$ be a Rota-Baxter {\bf k}-algebra of weight $\lambda$, and $\{(M_{i},p_{i})~|~i \in I \}$ be a family of left $\rbeta$-modules. Then the $\rbeta$-module $\bigoplus_{i\in I}(M_{i},p_{i})$ is flat if and only if each $\rbeta$-module $(M_{i},p_{i})$ is flat.
\end{prop}
\begin{proof}
Let $(L,p_{L})$ and $(N,p_{N})$ be two right $\rbeta$-modules, and let $\theta: (L,p_{L})\longrightarrow (N,p_{N})$ be a monomorphic $\rbeta$-module homomorphism.

Suppose that each left $\rbeta$-module $(M_{i},p_{i})$ is flat. Then each group homomorphism
$$
\theta \otimes id_{M_{i}}: L\otimes_{\rbeta} M_{i}\longrightarrow N\otimes_{\rbeta} M_{i}
$$
is injective. By Lemma~\ref{dir-sum-injective}, the homomorphism
$$\bigoplus_{i\in I}(\theta \otimes id_{M_{i}}): \bigoplus_{i\in I}(L\otimes_{\rbeta} M_{i})\longrightarrow \bigoplus_{i\in I}(N\otimes_{\rbeta} M_{i})$$
is also injective. Thus the $\rbeta$-module $\bigoplus_{i\in I}(M_{i}, p_{i})$ is flat by Lemma~\ref{dir-otim-iso}.

Conversely, suppose that the left $\rbeta$-module $\bigoplus_{i\in I}(M_{i},p_{i})$ is flat. Then the group homomorphism
$$\theta \otimes id_{(\bigoplus_{i\in I}M_{i})}: L\otimes_{\rbeta} \Big(\bigoplus_{i\in I}M_{i}\Big)\longrightarrow N\otimes_{\rbeta} \Big(\bigoplus_{i\in I}M_{i}\Big)$$
is an injective map. By Lemma~\ref{dir-otim-iso}, we conclude that
$$\bigoplus_{i\in I}(L\otimes_{\rbeta} M_{i})\longrightarrow \bigoplus_{i\in I}(N\otimes_{\rbeta} M_{i})$$
is also an injective map. By Lemma~\ref{dir-sum-injective}, for each $i\in I$, the map
$$
\theta \otimes id_{M_{i}}: L\otimes_{\rbeta} M_{i}\longrightarrow N\otimes_{\rbeta} M_{i}
$$
is injective. Thus each left $\rbeta$-module $(M_{i},p_{i})$ is flat.
\end{proof}

\begin{theorem}Let $\rbeta$ be a Rota-Baxter {\bf k}-algebra of weight $\lambda$.
Every free left $\rbeta$-module is flat.
\mlabel{thm:flat}
\end{theorem}

\begin{proof}
Let $(\calm_R(X)/I_X,p)$ be the free left $\rbeta$-module on $X$ defined in Theorem~\mref{thm:freem}. We just need to prove that if $(N',P_{N'})\to (N,p_N)$ is a monomorphism of right $\rbeta$-modules, then $N'\ot_{\rbeta} (\calm_R(X)/I_X)\to N\ot_{\rbeta}(\calm_R(X)/I_X)$ is a monomorphism of abelian groups.
We prove this in several steps.

First, for each right $\rbeta$-module $(M, p)$ and singleton $X=\{x\}$, we have
$$M\ot_{\rbeta}(\calm_R(\{x\})/ I_{\{x\}})\cong M.$$
This can be achieved by defining
\begin{eqnarray*}
f: M\ot_{\rbeta} \Big(\Big(\bigoplus_{n\geq 1}R^{\otimes n}\Big)x/I_{\{x\}}\Big) \longrightarrow M, \
v\ot ((r_1\ot \cdots r_n)x+I_{\{x\}}) \mapsto vr_1\cdots r_n,
\end{eqnarray*}
and
\begin{eqnarray*}
f': M \longrightarrow M\ot_{\rbeta} \Big(\Big(\bigoplus_{n\geq 1}R^{\otimes n}\Big)x/I_{\{x\}}\Big), \
v \mapsto v\ot (x+I_{\{x\}}).
\end{eqnarray*}
Then it is easy to check that $f\circ f'=id_{M}$ and $f'\circ f=id_{M\ot_{\rbeta}((\oplus_{n\geq 1}R^{\otimes n})x/I_{\{x\}})}$. Thus the maps $f, f'$ are bijective.

Second, for any set $X$, by noting $\calm_R(X)/I_X \cong \bigoplus_{x\in X}\calm_R(\{x\})/I_{\{x\}})$, we have
\begin{eqnarray*}
M\ot_{\rbeta} (\calm_R(X)/I_X) &\cong& M\ot_{\rbeta} \Big(\bigoplus_{x\in X}\calm_R(\{x\})/I_{\{x\}}\Big)\\ &\cong& \bigoplus_{x\in X}(M\ot_{\rbeta} \calm_R(\{x\})/I_{\{x\}}) \quad (\text{by Lemma~\ref{dir-otim-iso}}) \\
&\cong& \bigoplus_{x\in X} M.
\end{eqnarray*}

Consequently, $N'\ot_{\rbeta} (\calm_R(X)/I_X)\cong \oplus_{x\in X} N'$, and $N\ot_{\rbeta}(\calm_R(X)/I_X)\cong \oplus_{x\in X} N$. Then by Lemma~\ref{dir-sum-injective}, the group homomorphism $\oplus_{x\in X} N'\longrightarrow \oplus_{x\in X} N$ is injective. So the free left $\rbeta$-module $(\calm_R(X)/I_X, p)$ is flat.
\end{proof}

\begin{lemma}\label{proj-free}
Let $\rbeta$ be a Rota-Baxter {\bf k}-algebra of weight $\lambda$.
Every projective left $\rbeta$-module is a direct summand of a free $\rbeta$-module.
\end{lemma}

\begin{proof}
Let $(M,p)$ be a projective $\rbeta$-module. Let $(F(M),p')$ denote the free $\rbeta$-module over the set $M$. The identity map $id_M: (M,p) \rightarrow (M,p)$, when taken as a set map, gives a $(R,P)$-module epimorphism $f: (F(M),p') \twoheadrightarrow (M,p)$ such that $f|_M=id_M$. On the other hand, treating $id_M$ as a $(R,P)$-module homomorphism, the projectivity of $(M, p)$ gives a $(R,P)$-module homomorphism $\beta: (M,p) \rightarrow (F(M),p')$ such that $f\circ \beta = id_{(M,p)}$.
This gives
$$
F(M)=\im\, \beta \oplus \ker f \cong M\oplus \ker f.
$$
Since $\beta$ and $f$ are $(R,P)$-module homomorphisms, this is a direct sum of $(R,P)$-modules.
\end{proof}

By Proposition~\ref{dir-sum-flat}, Theorem~\ref{thm:flat} and Lemma~\ref{proj-free}, we obtain the following conclusion.

\begin{theorem}
Let $\rbeta$ be a Rota-Baxter {\bfk}-algebra of weight $\lambda$.
Then every projective left $\rbeta$-module is flat.
\mlabel{thm:enouf}
\end{theorem}

Theorem~\mref{thm:enouf} shows that there are enough flat Rota-Baxter modules, allowing us to define the Tor functors.

\smallskip

\noindent {\bf Acknowledgements}: This work was supported by the National Natural Science Foundation of China (Grant No.~11371177, 11371178 and 11501466). X. Gao thanks Rutgers University at Newark for its hospitality during his visit in 2015-2016.

\end{document}